\documentclass[a4paper,12pt]{amsart}
\usepackage{amscd,amsmath,amsthm,amssymb}
\usepackage{pstricks}
\usepackage{latexsym,amsbsy,mathrsfs}
\usepackage{xy}
\usepackage{xypic}
\usepackage{pgf,tikz}

\newtheorem{thm}{Theorem}[section]
\newtheorem{lem}[thm]{Lemma}
\newtheorem{cor}[thm]{Corollary}

\theoremstyle{definition}
\newtheorem{example}[thm]{Example}

\newtheorem{defn}[thm]{Definition}

\numberwithin{equation}{thm}

\def\hom{\mbox{\rm Hom}}

\begin{document}
\title[QUOTIENTS OF ONE-SIDED TRIANGULATED CATEGORIES]
{QUOTIENTS OF ONE-SIDED TRIANGULATED CATEGORIES BY RIGID
SUBCATEGORIES AS MODULE CATEGORIES}

\author{Zengqiang Lin and Yang Zhang}
\address{ School of Mathematical sciences, Huaqiao University,
Quanzhou\quad 362021,  China.} \email{lzq134@163.com}
\address{ School of Mathematical sciences, Huaqiao University,
Quanzhou\quad 362021,  China.} \email{zy100912@hqu.edu.cn}

\thanks{Supported by the NSF of China (Grants 11126331, 11101084)}



\begin{abstract}
We prove that some subquotient categories of one-sided triangulated
categories are abelian. This unifies a result by Iyama-Yoshino in
the case of triangulated categories and a result by Demonet-Liu in
the case of exact categories.
\end{abstract}

\maketitle

\section{Introduction}

\par
Cluster tilting theory gives a way to construct abelian categories
from some triangulated categories. Let $H$ be a hereditary algebra
over a field $k$, and $\mathcal{C}$ be the cluster category defined
in [1] as the factor category $D^b(\mbox{mod}H)/\tau^{-1}\Sigma$,
where $\tau$ and $\Sigma$ be the  Auslander-Reiten translation
 and shift functor of $D^b(\mbox{mod}H)$ respectively.
For  a cluster tilting object $T$ in $\mathcal{C}$, Buan, Marsh and
Reiten [2] showed that $\mathcal{C}/{\text{add}\tau T}\cong\text{mod
}\text{End}_{\mathcal{C}}(T)^{\text{op}}$.  Keller and Reiten [3]
generalized this result in the case of 2-Calabi-Yau triangulated
categories by showing that ${\mathcal{C}/\Sigma\mathcal{T}}\cong
\mbox{mod}{\mathcal{T}}$, where ${\mathcal{T}}$ is a cluster tilting
subcategory of $\mathcal{C}$. A general framework for cluster
tilting is set up by Koenig and Zhu. They [4] showed that any
quotient of a triangulated category modulo a cluster titling
subcategory carries an abelian structure . Let $\mathcal{C}$ be a
triangulated category and $\mathcal{M}$ be a rigid subcategory, i.e.
Hom$_{\mathcal{C}}(\mathcal{M},\Sigma\mathcal{M})=0$. Iyama and
Yoshino [5] showed that
${\mathcal{M}\ast\Sigma\mathcal{M}/\Sigma\mathcal{M}}\cong\mbox{mod}\mathcal{M}$.
In particular, if $\mathcal{M}$ is a cluster tilting subcategory,
then $\mathcal{M}\ast\Sigma\mathcal{M}=\mathcal{C}$, thus the work
generalized some former results in [2,3,4].

Recently, Cluster tilting theory is also permitted to construct
abelian categories from some exact categories. Let $\mathcal{B}$ be
an exact category with enough projectives  and $\mathcal{M}$ be a
cluster tilting subcategory. Demonet and Liu [6] showed that
$\mathcal{B}/\mathcal{M}\cong\mbox{mod}\underline{\mathcal{M}}$,
which generalized the work of Koenig and Zhu in the case of
Frobenius categories.

The main aim of this article is to unify the work of Iyama-Yoshino
and Demonet-Liu, and give a framework for construct abelian
categories from triangulated categories and exact categories. Our
setting is one-sided  triangulated category, which is a natural
generalization of triangulated category.  Left and right
triangulated categories were defined by Beligiannis and Marmaridsis
in [7]. For details and more information on
 one-sided  triangulated
categories we refer to [7-9].

%Recall that $\mathcal{M}$ is called a rigid subcategory of a
%triangulated category $\mathcal{C}$, if
%Hom$_{\mathcal{C}}(\mathcal{M},\Sigma\mathcal{M})=0$, where $\Sigma$
%is the suspension functor of $\mathcal{C}$. Similarly, recall that
%$\mathcal{T}$ is called a rigid subcategory of an exact category
%$\mathcal{B}$, if
%Ext$^{1}_{\mathcal{B}}(\mathcal{T},\mathcal{T})=0$. By observing, we
%define the notion of rigid subcategories of one-sided triangulated
%categories (seen in Definition 3.4 and 4.5), which generalizes the
%two notions of rigid subcategories of triangulated categories and of
%exact categories. The aim of this paper is to unify the two
%constructions of the subquotient triangulated categories by
%Iyama-Yoshino in [4] and of the subquotient exact categories by
%Demonet-Liu in [5].

The paper is organized as follows. In Section 2, we review some
basic material on module categories over $k$-linear categories and
quotient categories etc. In Section 3, we prove that some
subquotient categories of right triangulated categories are module
categories, which unifies the Proposition 6.2 in [4] and the Theorem
3.5 in [5]. In Section 4, we prove that some subquotient categories
of left triangulated categories are module categories, which unifies
the Proposition 6.2 in [4] and the Theorem 3.2 in [5]. And we will
see that the case of right triangulated categories and the case of
left triangulated categories are not  dual.

\section{Preliminaries}

Throughout this paper, $k$ denotes a field. When we say that
$\mathcal{C}$ is a category, we always assume that $\mathcal{C}$ is
a Hom-finite Krull-Schmidt $k$-linear category. For a subcategory
$\mathcal{M}$ of category $\mathcal{C}$, we mean $\mathcal{M}$ is an
additive full subcategory of $\mathcal{C}$ which is closed under
taking direct summands. Let $f:X\rightarrow Y$, $g:Y\rightarrow Z$
be morphisms in $\mathcal{C}$, we denote by $gf$ the composition of
$f$ and $g$, and $f_{\ast}$ the morphism $\hom_{C}(M,f):
\mbox{Hom}_\mathcal{C}(M,X)\rightarrow \hom_{\mathcal{C}}(M,Y)$ for
any $M\in\mathcal{C}$.

%For an object $X$ of $\mathcal{C}$, denoted by $X\in\mathcal{C}$.
%For a subcategory $\mathcal{T}$ of $\mathcal{C}$, denoted by
%$\mathcal{T}\subset\mathcal{C}$.

Let $\mathcal{C}$ be a category and $\mathcal{X}$ be a subcategory
of $\mathcal{C}$. A right $\mathcal{X}$-approximation of $C$ in
$\mathcal{C}$ is a map $f: X\rightarrow C$, with $X$ in
$\mathcal{X}$, such that for all objects $X'$ in $\mathcal{X}$, the
sequence
Hom$_{\mathcal{C}}(X',X)\rightarrow$Hom$_{\mathcal{C}}(X',C)\rightarrow0$
is exact. If for any object $C\in\mathcal{C}$, there exists a right
$\mathcal{X}$-approximation $f:X\rightarrow C$, then $\mathcal{X}$
is called a contravariantly finite subcategory of $\mathcal{C}$.
Dually we have the notions of left $\mathcal{X}$-approximation and
covariantly finite subcategory. $\mathcal{X}$ is called functorially
finite if $\mathcal{X}$ is contravariantly finite and covariantly
finite.

Let $\mathcal{C}$ be a category. A pseudokernel of a morphism $v:
V\rightarrow W$ in $\mathcal{C}$ is a morphism $u: U\rightarrow V$
such that $vu=0$ and if $u':U'\rightarrow V$ is a morphism such that
$vu'=0$, there exists $f: U'\rightarrow U$ such that $u'=uf$.
Pseudocokernels are defined dually.

Let $\mathcal{C}$ be a category. A $\mathcal{C}$-module is a
contravariant $k$-linear functor $F:C\rightarrow$ Mod$k$. Then
$\mathcal{C}$-modules form an abelian category Mod$\mathcal{C}$. By
Yoneda's lemma, representable functors Hom$_{\mathcal{C}}(-,C)$ are
projective objects in Mod$\mathcal{C}$. We denote by mod
$\mathcal{C}$ the subcategory of Mod$\mathcal{C}$ consisting of
finitely presented $\mathcal{C}$-modules. One can easily check that
mod$\mathcal{C}$ is closed under cokernels and extensions in
Mod$\mathcal{C}$. Moreover, mod$\mathcal{C}$ is closed under kernels
in Mod$\mathcal{C}$ if and only if $\mathcal{C}$ has pseudokernels.
In this case, mod$\mathcal{C}$ forms an abelian category (see [10]).
For example, if $\mathcal{C}$ is a contravariantly finite
subcategory of a triangulated category, then mod$\mathcal{C}$ forms
an abelian category.

Let $\mathcal{C}$ be an additive category and $\mathcal{B}$ be a
subcategory of $\mathcal{C}$. For any two objects
$X,Y\in\mathcal{C}$, denote by $\mathcal{B}(X,Y)$ the additive
subgroup of Hom$_{\mathcal{C}}(X,Y)$ such that for any morphism
$f\in\mathcal{B}(X,Y)$, $f$ factors through some object in
$\mathcal{B}$. We denote by $\mathcal{C}/\mathcal{B}$ the quotient
category whose objects are objects of $\mathcal{C}$ and whose
morphisms are elements of Hom$_{\mathcal{C}}(M,N)/\mathcal{B}(M,N)$.
The projection functor
$\pi:\mathcal{C}\rightarrow\mathcal{C}/\mathcal{B}$ is an additive
functor satisfying $\pi(\mathcal{B})=0$, and for any additive
functor $F:\mathcal{C}\rightarrow\mathcal{D}$ satisfying
$F(\mathcal{B})=0$, there exists a unique additive functor
$G:\mathcal{C}/\mathcal{B}\rightarrow\mathcal{D}$ such that
$F=G\pi$. We have the following two easy and useful facts.

\begin{lem}
Let $F: \mathcal{C}\rightarrow \mathcal{D}$ be an additive functor.
If $F$ is full and dense, and there exists a subcategory
$\mathcal{B}$ of $\mathcal{C}$ such that any morphism
$f:X\rightarrow Y$ in $\mathcal{C}$ with $F(f)=0$ factors through
some object in $\mathcal{B}$, then $F$ induces an equivalence
$\mathcal{C}/\mathcal{B}\cong\mathcal{D}$.
\end{lem}

\begin{lem}
Let $\mathcal{A}$ be an additive category,  $\mathcal{B}$ and
$\mathcal{C}$ be two subcategories of $\mathcal{A}$ with
$\mathcal{C}\subset\mathcal{B}$. Then there exists an equivalence of
categories
$(\mathcal{A}/\mathcal{C})/(\mathcal{B}/\mathcal{C})\cong\mathcal{A}/\mathcal{B}$.
\end{lem}

\begin{proof}
Let
$\pi_{\mathcal{B}}:\mathcal{A}\rightarrow{\mathcal{A}/\mathcal{B}}$
and
$\pi_{\mathcal{C}}:\mathcal{A}\rightarrow{\mathcal{A}/\mathcal{C}}$
be the projection functors.
%For any morphism $g$ in $\mathcal{A}$,
%we denote $[g]=\pi_{\mathcal{C}}(g)$.
Note that $\mathcal{C}\subset\mathcal{B}$, we have
$\pi_{\mathcal{B}}(\mathcal{C})=0$, thus there exists a unique
functor
$F:\mathcal{A}/\mathcal{C}\rightarrow\mathcal{A}/\mathcal{B}$ such
that $F\pi_{\mathcal{C}}=\pi_{\mathcal{B}}$. Since
$\pi_{\mathcal{B}}$ is full and dense, $F$ is full and dense too.

Let $f:X\rightarrow Y$ be a morphism in $\mathcal{A}$ such that
$F(\pi_{\mathcal{C}}(f))=0$, that is $\pi_{\mathcal{B}}(f)=0$. Then
$f$ factors through some object in $\mathcal{B}$, thus
$\pi_{\mathcal{C}}(f)$ factors through some object in
$\mathcal{B}/\mathcal{C}$. According to Lemma 2.1, we
 have
an equivalence of categories
$(\mathcal{A}/\mathcal{C})/(\mathcal{B}/\mathcal{C})\xrightarrow{\sim}\mathcal{A}/\mathcal{B}$.
\end{proof}

\section{Subquotient categories of right triangulated categories}

Firstly, we recall some basics on right triangulated categories from
[8].

\begin{defn}
A right triangulated category is a triple
$(\mathcal{C},\Sigma,\rhd)$, or simply $\mathcal{C}$, where:

(a)\quad $\mathcal{C}$ is an additive category.

(b)\quad $\Sigma:\mathcal{C}\rightarrow\mathcal{C}$ is an additive
functor, called the shift functor of $\mathcal{C}$.

(c)\quad $\rhd$ is a class of sequences of three morphisms of the
form $U\xrightarrow{u}V\xrightarrow{v}W\xrightarrow{w}\Sigma U$,
called right triangles, and satisfying the following axioms:

(RTR0)\quad If
$U\xrightarrow{u}V\xrightarrow{v}W\xrightarrow{w}\Sigma U$ is a
right triangle, and
$U'\xrightarrow{u'}V'\xrightarrow{v'}W'\xrightarrow{w'}\Sigma U'$ is
a sequence of morphisms such that there exists a commutative diagram
in $\mathcal{C}$
$$\xymatrix{
U\ar[r]^{u}\ar[d]^{f} & V\ar[r]^{v}\ar[d]^{g} &
W\ar[r]^{w}\ar[d]^{h} & \Sigma U\ar[d]^{\Sigma f} \\
U'\ar[r]^{u'} & V'\ar[r]^{v'} & W'\ar[r]^{w'} & \Sigma U', }
$$
where $f,g,h$ are isomorphisms, then
$U'\xrightarrow{u'}V'\xrightarrow{v'}W'\xrightarrow{w'}\Sigma U'$ is
also a right triangle.

(RTR1)\quad For any $U\in\mathcal{C}$, the sequence
$0\xrightarrow{}U\xrightarrow{1_{U}}U\xrightarrow{}0$ is a right
triangle. And for any morphism $u:U\rightarrow V$ in $\mathcal{C}$,
there exists a right triangle
$U\xrightarrow{u}V\xrightarrow{v}W\xrightarrow{w}\Sigma U$.

(RTR2)\quad If
$U\xrightarrow{u}V\xrightarrow{v}W\xrightarrow{w}\Sigma U$ is a
right triangle, then so is $V\xrightarrow{v}W\xrightarrow{w}\Sigma
U\xrightarrow{-\sum u}\Sigma V$.

(RTR3)\quad For any two right triangles
$U\xrightarrow{u}V\xrightarrow{v}W\xrightarrow{w}\Sigma U$ and
$U'\xrightarrow{u'}V'\xrightarrow{v'}W'\xrightarrow{w'}\Sigma U'$,
and any two morphisms $f:U\rightarrow U'$, $g:V\rightarrow V'$ such
that $gu=u'f$, there exists $h:W\rightarrow W'$ such that the
following  diagram is commutative
$$\xymatrix{
U\ar[r]^{u}\ar[d]^{f} & V\ar[r]^{v}\ar[d]^{g} &
W\ar[r]^{w}\ar@{-->}[d]^{h} & \Sigma U\ar[d]^{\Sigma f} \\
U'\ar[r]^{u'} & V'\ar[r]^{v'} & W'\ar[r]^{w'} & \Sigma U'. }
$$

(RTR4)\quad For any two right triangles
$U\xrightarrow{u}V\xrightarrow{v}W\xrightarrow{w}\Sigma U$ and
$U'\xrightarrow{u'}U\xrightarrow{v'}W'\xrightarrow{w'}\Sigma U'$,
there exists a commutative diagram
$$\xymatrix{
U'\ar[r]^{u'}\ar@{=}[d] & U\ar[r]^{v'}\ar[d]^{u} & W'\ar[r]^{w'}\ar[d]^{f} & \Sigma U\ar@{=}[d] \\
U'\ar[r]^{u\cdot u'}    & V\ar[r]^{p}\ar[d]^{v}  & V'\ar[r]^{q}\ar[d]^{g}  & \Sigma U'          \\
                        & W\ar@{=}[r]\ar[d]^{w}  & W\ar[d]^{\Sigma v'\cdot w}                    \\
                        & \Sigma U\ar[r]^{\Sigma v'}         & \Sigma W',
}
$$
where the second row and the third column are right triangles.
\end{defn}

\begin{example}
A triangulated category $\mathcal{C}$ is a  right triangulated
category, where the shift functor $\Sigma$ is an equivalence. In
this case, right triangles in $\mathcal{C}$ are called triangles.
\end{example}

\begin{example} (cf.[7,11])
Let $\mathcal{B}$ be an exact category which contains enough
injectives. The subcategory of injectives is denoted by
$\mathcal{I}$. Then the quotient category
$\overline{\mathcal{B}}=\mathcal{B}/\mathcal{I}$ is a right
triangulated category. For any morphism
$f\in\mbox{Hom}_\mathcal{B}(X,Y)$, we denote its image in
Hom$_{\overline{\mathcal{B}}}(X,Y)$ by $\overline{f}$. Let us recall
the definitions of the shift functor $\Sigma$ and of the
distinguished right triangles. For any $X\in\mathcal{B}$, there is a
short exact sequence $0\rightarrow X \xrightarrow{i_X}
I_X\xrightarrow {d_X}C_X\rightarrow 0$ with $I_X\in\mathcal{I}$. For
any morphism $f:X\rightarrow Y$, we have the following commutative
diagram with exact rows
$$\xymatrix{
0\ar[r] & X\ar[d]_{f}\ar[r]^{i_{X}} &
I_{X}\ar[r]^{d_{X}}\ar[d]_{i_{f}} & C_{X}\ar[r]\ar[d]_{c_{f}}
& 0 \\
0\ar[r] & Y\ar[r]^{i_{Y}} & I_{Y}\ar[r]^{d_{Y}} & C_{Y}\ar[r] & 0,}
$$
where $I_{X},I_{Y}\in\mathcal{I}$. Define $\Sigma(X)=C_X$ and
$\Sigma\overline{f}=\overline{c_{f}}$. We can show that the functor
$\Sigma$ is well defined.
%Then we can easily prove that
%$\Sigma:\overline{\mathcal{B}}\rightarrow\overline{\mathcal{B}}$ is
%an additive functor and
%$\overline{\Sigma\mathcal{M}}=\Sigma\overline{\mathcal{M}}$.
For any morphism $f:X\rightarrow Y$, we have the following
commutative diagram with exact rows
$$\xymatrix{
0\ar[r] & X\ar[d]_{f}\ar[r]^{i_{X}} &
I_{X}\ar[r]^{d_{X}}\ar[d]_{i_{f}} & C_X\ar[r]\ar@{=}[d]
& 0 \\
0\ar[r] & Y\ar[r]^{g} & Z\ar[r]^{h} & C_X\ar[r] & 0,}
$$
where $Z$ is the pushout of $f$ and $i_X$. Then
$X\xrightarrow{\overline{f}}Y\xrightarrow{\overline{g}}Z\xrightarrow{\overline{h}}\Sigma
X$, or equivalently $X\xrightarrow{\tiny\left(
                                       \begin{array}{c}
                                         \overline{f} \\
                                         \overline{i_X} \\
                                       \end{array}
                                     \right)}Y\oplus
I_X\xrightarrow{(\overline{g},-\overline{i_f})}Z\xrightarrow{\overline{h}}\Sigma
X$ is a distinguished  right triangle. In this case, there is a
short exact sequence $0\rightarrow X\xrightarrow{\tiny\left(
                                       \begin{array}{c}
                                         f \\
                                         i_X \\
                                       \end{array}
                                     \right)
}Y\oplus I_X\xrightarrow{(g,-i_f)}Z\rightarrow0$. And we have the
following commutative diagram of short exact sequences
$$\xymatrix{
0\ar[r] & X\ar@{=}[d]\ar[rr]^{\tiny\left(
                      \begin{array}{c}
                        f \\
                        i_X \\
                      \end{array}
                    \right)}& & Y\oplus I_{X}\ar[d]_{(0,1)}\ar[rr]^{(g,-i_{f})}
& & Z\ar[d]_{-h}\ar[r] & 0\\
0\ar[r] & X\ar[rr]^{i_X}& & I_X\ar[rr]^{d_X} & & \Sigma X\ar[r] & 0.
}
$$
So a distinguished
right triangle in $\overline{\mathcal{B}}$ give rise to a short
exact sequence in $\mathcal{B}$. On the other hand, Let
$0\rightarrow X\xrightarrow{f}Y\xrightarrow {g}Z\rightarrow0$ be a
short exact sequence in $\mathcal{B}$, then we have the following
commutative diagram with exact rows
$$\xymatrix{
0\ar[r] & X\ar@{=}[d]\ar[r]^{f} & Y\ar[r]^{g}\ar[d]_{i_{Y}} &
Z\ar[r]\ar[d]^{h}
& 0 \\
0\ar[r] & X\ar[r]^{i_Y} & I_Y\ar[r]^{p} & \Sigma X\ar[r] & 0,}$$
where $I_Y\in\mathcal{I}$, and
$X\xrightarrow{\overline{f}}Y\xrightarrow{\overline{g}}Z\xrightarrow{-\overline{h}}\Sigma
X$ is a right triangle in $\overline{\mathcal{B}}$ [11]. Thus, a
short exact sequence in $\mathcal{B}$ give rise to a right triangle
in $\overline{\mathcal{B}}$.
\end{example}
%Denote by $\rhd$ the class of sequences of
%three morphisms which are equivalent to the standard right
%triangles. Hence $(\overline{\mathcal{B}},\Sigma,\rhd)$ is a right
%triangulated category.

%Let
%$U\xrightarrow{\overline{u}}V\xrightarrow{\overline{v}}W\xrightarrow{\overline{w}}\Sigma
%U$ be a standard right triangle in $\overline{\mathcal{B}}$, we have
%the following commutative diagram with exact rows in $\mathcal{B}$:
%$$\xymatrix{
%0\ar[r] & U\ar@{=}[d]\ar[r]^{u'} & V\ar[r]^{v'}\ar[d] &
%W\ar[r]\ar[d]^{w'}
%& 0 \\
%0\ar[r] & U\ar[r] & I_{U}\ar[r] & \Sigma U\ar[r] & 0,}
%$$
%where
%$\overline{u}=\overline{u'}$,$\overline{v}=\overline{v'}$,$\overline{w}=\overline{w'}$.

The following lemma can be found in [7].

\begin{lem}
Let $\mathcal{C}$ be a right triangulated category, and
$U\xrightarrow{u}V\xrightarrow{v}W\xrightarrow{w}\Sigma U$ be a
right triangle.

(a) $v$ is a pseudocokernel of $u$, and $w$ is a pseudocokernel of
$v$.

(b) If $\Sigma$ is fully faithful, then $u$ is a pseudokernel of
$v$, and $v$ is a pseudokernel of $w$.

%(c) For any object $X$ of $\mathcal{C}$, there exists an exact
%sequence $$\cdots\rightarrow\text{Hom}_{\mathcal{C}}(\Sigma
%V,X)\xrightarrow{-(\Sigma u)^{\ast}}\text{Hom}_{\mathcal{C}}(\Sigma
%U,X)\xrightarrow{w^{\ast}}$$$$
%\text{Hom}_{\mathcal{C}}(W,X)\xrightarrow{v^{\ast}}\text{Hom}_{\mathcal{C}}(V,X)\xrightarrow{u^{\ast}}
%\text{Hom}_{\mathcal{C}}(U,X).$$
\end{lem}

%\begin{proof}
%By (RTR3), there exists a morphism $h:W\rightarrow W'$ making the
%diagram above commutative. $\forall X\in\mathcal{C}$, by Proposition
%2.1.3 and Five Lemma, we have a isomorphism
%$h^{\ast}$:Hom$_{\mathcal{C}}(W',X)\rightarrow$Hom$_{\mathcal{C}}(W,X)$.
%By Yoneda's Lemma, $h$ is a isomorphism.
%\end{proof}

\begin{defn}
Let $\mathcal{C}$ be a right triangulated category. A subcategory
$\mathcal{M}$ of $\mathcal{C}$ is called a rigid subcategory if
Hom$_{\mathcal{C}}(\mathcal{M},\Sigma\mathcal{M})=0$.
\end{defn}

Let $\mathcal{M}$ be a rigid subcategory of $\mathcal{C}$. Denote by
$\mathcal{M}\ast\Sigma\mathcal{M}$ the subcategory of $\mathcal{C}$
consisting of all such $X\in\mathcal{C}$ with right triangles
$M_{0}\rightarrow M_{1}\rightarrow X\rightarrow\Sigma M_{0},$ where
$M_{0},M_{1}\in\mathcal{M}$.
%Since $\mathcal{M}$ is rigid,
%$\mathcal{M}\ast\Sigma\mathcal{M}$ is closed under direct summands.

Now we can state the main theorem of this section.

%Since $\mathcal{M}$ is a rigid subcategory,
%$F(\Sigma\mathcal{M})=0$. Hence by the universal property of the
%projection functor
%$\pi:\mathcal{M}\ast\Sigma\mathcal{M}\rightarrow\mathcal{M}\ast\Sigma\mathcal{M}\big{/}\Sigma\mathcal{M}$,
%$F$ can induce a functor
%$\overline{F}:\mathcal{M}\ast\Sigma\mathcal{M}\big{/}\Sigma\mathcal{M}\rightarrow
%\text{Mod }\mathcal{M}$.

\begin{thm}
Let $\mathcal{C}$ be a right triangulated category, and
$\mathcal{M}$ be a rigid subcategory of $\mathcal{C}$ satisfying:

(RC1)\quad$\Sigma$ is fully faithful when it is restricted to
$\mathcal{M}$.

(RC2)\quad For any two objects $M_{0},M_{1}\in\mathcal{M}$, if
$M_{0}\xrightarrow{f}M_{1}\xrightarrow{g}X\xrightarrow{h}\Sigma
M_{0}$ is a right triangle in $\mathcal{C}$, then $g$ is a right
$\mathcal{M}$-approximation of $X$.

Then there exists an equivalence of categories
$\mathcal{M}\ast\Sigma\mathcal{M}\big{/}\Sigma\mathcal{M}\cong\mbox{mod}\mathcal{M}$.
\end{thm}

 Before prove the theorem, we prove the lemma as follow.

\begin{lem}
Under the same assumption as in Theorem 3.6, for any right triangle
$M_{0}\xrightarrow{f}M_{1}\xrightarrow{g}X\xrightarrow{h}\Sigma
M_{0}$ where $M_0, M_1\in\mathcal{M}$, there is an exact sequence in
$\text{Mod} \mathcal{M}$
$$\text{Hom}
_{\mathcal{M}}(-,M_{0})\xrightarrow{\text{Hom}_{\mathcal{M}}(-,f)}
\mbox{Hom}_{\mathcal{M}}(-,M_{1})\xrightarrow{\text{Hom}_{\mathcal{C}}(-,g)}
\text{Hom}_{\mathcal{C}}(-,X)|_{\mathcal{M}}\rightarrow0.$$ Thus,
$\text{Hom}_{\mathcal{C}}(-,X)|_{\mathcal{M}}\in\text{mod}\mathcal{M}$.
\end{lem}

\begin{proof}
Let $M_{0}\xrightarrow{f}M_{1}\xrightarrow{g} X\xrightarrow{h}\Sigma
M_{0}$ be a right triangle with $M_{0},M_{1}\in\mathcal{M}$. For any
$M\in\mathcal{M}$, we claim that the following sequence is exact
$$\text{Hom}_{\mathcal{C}}(M,M_{0})\xrightarrow{f_{\ast}}\text{Hom}
_{\mathcal{C}}(M,M_{1})\xrightarrow{g_{\ast}}\text{Hom}
_{\mathcal{C}}(M,X)\rightarrow0.\quad(\star)$$ In fact, by Lemma 3.4
(a), we have $ gf=0$, hence Im$f_{\ast}\subseteq$Ker$g_{\ast}$. For
any $ t\in$Ker$g_{\ast}$, we have the following commutative diagram
of right triangles by (RTR3)
$$\xymatrix{
M\ar[r]\ar[d]^{t} & 0\ar[r]\ar[d] & \Sigma M\ar[r]^{-\Sigma1_{M}}\ar@{-->}[d]^{m'} & \Sigma M\ar[d]^{\Sigma t} \\
M_{1}\ar[r]^{g}           & X\ar[r]^{h}   & \Sigma
M_{0}\ar[r]^{-\Sigma f} & \Sigma M_{1}.}
$$
Since $\Sigma|_{\mathcal{M}}$ is full, there exists a morphism
$m:M\rightarrow M_{0}$ such that $m'=\Sigma m$, so $\Sigma t=\Sigma
(fm)$. Since $\Sigma|_{\mathcal{M}}$ is faithful,
$t=fm=f_{\ast}(m)\in$Im$f_{\ast}$, then Im$f_{\ast}\supseteq$Ker$g_{
\ast}$. Hence Im$f_{\ast}$=Ker$g_{ \ast}$. On the other hand, by
(RC2), $g_{\ast}$ is surjective. So ($\star$) is exact. Since $M$ is
arbitrary in $\mathcal{M}$, there exists an exact sequence
$$\mbox{Hom}
_{\mathcal{M}}(-,M_{0})\xrightarrow{\text{Hom}_{\mathcal{M}}(-,f)}
\mbox{Hom}_{\mathcal{M}}(-,M_{1})\xrightarrow{\text{Hom}_{\mathcal{C}}(-,g)}
\text{Hom}_{\mathcal{C}}(-,X)|_{\mathcal{M}}\rightarrow0.$$
%so $F(X)\in$ mod$\mathcal{M}$.
\end{proof}

\noindent{\bf Proof of Theorem 3.6.} By Lemma 3.7, we have an
additive functor $F:\mathcal{M}\ast\Sigma\mathcal{M}\rightarrow
\text{Mod}\mathcal{M}$, which is defined by
$F(X)=\text{Hom}_{\mathcal{C}}(-,X)|_{\mathcal{M}}$.

%for any $X\in\mathcal{M}\ast\Sigma\mathcal{M}$.

%By the lemma above, $\overline{F}$ can be seen as a functor from
%$\mathcal{M}\ast\Sigma\mathcal{M}\big{/}\Sigma\mathcal{M}$ to mod
%$\mathcal{M}$. In the next, we give the proof of Theorem 3.5.

%\begin{proof}

%Let us prove that
%$\overline{F}:\mathcal{M}\ast\Sigma\mathcal{M}\big{/}\Sigma\mathcal{M}\rightarrow
%\text{mod }\mathcal{M}$ is dense and fully faithful.

%$\forall X,Y\in\mathcal{M}\ast\Sigma\mathcal{M}$, there exist two
%right triangles
%$M_{0}\xrightarrow{f_{1}}M_{1}\xrightarrow{g_{1}}X\xrightarrow{h_{1}}\Sigma
%M_{0}$ and
%$N_{0}\xrightarrow{f_{2}}N_{1}\xrightarrow{g_{2}}X\xrightarrow{h_{2}}\Sigma
%M_{0}$, where $M_{0},M_{1},N_{0},N_{1}\in\mathcal{M}$.

Firstly, we show that $F$ is dense.

For any object $G\in$ mod$\mathcal{M}$, there exists an exact
sequence $$\mbox{Hom}_{\mathcal{M}}(-,M')\xrightarrow{\alpha}
\mbox{Hom}_{\mathcal{M}}(-,M'')\rightarrow G\rightarrow0$$ with
$M',M''\in\mathcal{M}$. By Yoneda's Lemma, there exists a morphism
$f:M'\rightarrow M''$ such that $\alpha$=Hom$_{\mathcal{M}}(-,f)$.
Then by (RTR1), there exists a right triangle
$M'\xrightarrow{f}M''\xrightarrow{g}Z\xrightarrow{h}\Sigma M'$. By
Lemma 3.7, there exists an exact sequence
Hom$_{\mathcal{M}}(-,M')\xrightarrow{\alpha}$
Hom$_{\mathcal{M}}(-,M'')\rightarrow F(Z)\rightarrow0$, thus
$G$=Coker$\alpha\cong F(Z)$. Hence $F$ is dense.

Secondly, we show that $F$ is full.

For any morphism $\beta:F(X)\rightarrow F(Y)$ in mod $\mathcal{M}$,
 because  Hom$_{\mathcal{M}}(-,M_{1})$
 is a projective object in mod$\mathcal{M}$, we have the following
commutative diagram with exact rows in Mod$\mathcal{M}$
$$\xymatrix{
\text{Hom}_{\mathcal{M}}(-,M_{0})\ar[rr]^{\text{Hom}_{\mathcal{M}}(-,f_{1})}\ar@{-->}[d]^{\gamma_{0}}&
& \text{Hom}_{\mathcal{M}}(-,M_{1})\ar[r]\ar@{-->}[d]^{\gamma_{1}} &
F(X)\ar[d]^{\beta}\ar[r] & 0 \\
\text{Hom}_{\mathcal{M}}(-,N_{0})\ar[rr]^{\text{Hom}_{\mathcal{M}}(-,f_{2})}&
& \text{Hom}_{\mathcal{M}}(-,N_{1})\ar[r]& F(Y)\ar[r] & 0.}
$$
By Yoneda's Lemma, for $i=0,1$, there exists a morphism
$m_{i}:M_{i}\rightarrow N_{i}$ such that
$\gamma_{i}$=Hom$_{\mathcal{M}}(-, m_{i})$ and
$m_{1}f_{1}$=$f_{2}m_{0}$. Hence by (RTR3) we have the following
commutative diagram of right triangles
$$\xymatrix{
M_{0}\ar[r]^{f_{1}}\ar[d]^{m_{0}} &
M_{1}\ar[r]^{g_{1}}\ar[d]^{m_{1}} &
X\ar[r]^{h_{1}}\ar@{-->}[d]^{s} & \Sigma M_{0}\ar[d]^{\Sigma m_{0}} \\
N_{0}\ar[r]^{f_{2}} & N_{1}\ar[r]^{g_{2}} & Y\ar[r]^{h_{2}} & \Sigma
N_{0}. }
$$
Then by Lemma 3.7, we have the following commutative diagram with
exact rows in Mod$\mathcal{M}$
$$\xymatrix{
\text{Hom}_{\mathcal{M}}(-,M_{0})\ar[rr]^{\text{Hom}_{\mathcal{M}}(-,f_{1})}\ar[d]^{\gamma_{0}}&
& \text{Hom}_{\mathcal{M}}(-,M_{1})\ar[r]\ar[d]^{\gamma_{1}} &
F(X)\ar[d]^{F(s)}\ar[r] & 0 \\
\text{Hom}_{\mathcal{M}}(-,N_{0})\ar[rr]^{\text{Hom}_{\mathcal{M}}(-,f_{2})}&
& \text{Hom}_{\mathcal{M}}(-,N_{1})\ar[r]& F(Y)\ar[r] & 0.}
$$
So $\beta=F(s)$. Hence $F$ is full.

At last, in order to show
$\mathcal{M}\ast\Sigma\mathcal{M}/\Sigma\mathcal{M}\cong\text{mod}\mathcal{M}$,
by Lemma 2.1 we only need to prove that any morphism $t:X\rightarrow
Y$ in $\mathcal{M}\ast\Sigma\mathcal{M}$ satisfying $F(t)=0$ factors
through some object in $\Sigma\mathcal{M}$.

In fact, let
$M_{0}\xrightarrow{f_{1}}M_{1}\xrightarrow{g_{1}}X\xrightarrow{h_{1}}\Sigma
M_{0}$ be a right triangle with $M_0, M_1\in\mathcal{M}$, then
$tg_{1}=0$ since $F(t)=0$. Thus by Lemma 3.4(a), $t$ factors through
$h_{1}$, so $t$ factors through $\Sigma M_{0}\in\Sigma\mathcal{M}$.
$\Box$

Applying Theorem 3.6, we can get the following two corollaries.

\begin{cor} ([4, Proposition 6.2])
Let $\mathcal{C}$ be a triangulated category with the shift functor
$\Sigma$ and $\mathcal{M}$ be a rigid subcategory of
 $ \mathcal{C}$. Then there exists an equivalence of categories
$\mathcal{M}\ast\Sigma\mathcal{M}\big{/}\Sigma\mathcal{M}\cong\text{mod}\mathcal{M}$.
\end{cor}

\begin{proof} Since the shift functor $\Sigma$ is
 an equivalence, we know that  $\Sigma|_{\mathcal{M}}$ is fully
faithful. Let $M_{0}\xrightarrow{f} M_{1}\xrightarrow
{g}X\xrightarrow{h}\Sigma
 M_{0}$ be a triangle in $\mathcal{C}$, where
$M_{0},M_{1}\in\mathcal{M}$. Since $\mathcal{M}$ is rigid, we know
that $g$ is a right $\mathcal{M}$-approximation of $X$ by Lemma
3.4(b). Thus, condition (RC1) and (RC2) hold.
 % Then $\mathcal{M}$ is a rigid
%subcategory of right triangulated category $\mathcal{C}$. Denote by
%$\mathcal{M}\ast\Sigma\mathcal{M}$ the subcategory of objects $X$ in
%$\mathcal{C}$ such that there exist right triangles , Since $\Sigma$
%is auto-equivalent and $\mathcal{C}$ has cohomological functors, we
%can easily prove that $\Sigma|_{\mathcal{M}}$ is fully faithful and
%$\mathcal{M}$ satisfies the condition (RC2) in Theorem 3.5.
\end{proof}

\begin{defn}
Let $\mathcal{B}$ be an exact category and $\mathcal{M}$ be a full
subcategory of $\mathcal{B}$.  $\mathcal{M}$ is called rigid if
Ext$^1_B(\mathcal{M},\mathcal{M})=0$.
\end{defn}

\begin{cor} ([6, Theorem 3.5])
Let $\mathcal{B}$ be an exact category which contains enough
injectives, and $\mathcal{M}$ be a rigid subcategory of
$\mathcal{B}$ containing all injectives. Denote by $\mathcal{I}$ the
subcategory of injectives, and by $\overline{\mathcal{M}}$ the
quotient category $\mathcal{M}/\mathcal{I}$. Denote by
$\mathcal{M}_{R}$ the subcategory of objects $X$ in $\mathcal{B}$
such that there exist short exact sequences $0\rightarrow
M_{0}\rightarrow M_{1}\rightarrow X\rightarrow0$, where
$M_{0},M_{1}\in\mathcal{M}$. Denote by $\Sigma\mathcal{M}$ the
subcategory of objects $Y$ in $\mathcal{B}$ such that there exist
short exact sequences $0\rightarrow M\rightarrow I\rightarrow
Y\rightarrow0$, where $M\in\mathcal{M}$, $I\in\mathcal{I}$. Then
$\mathcal{M}_{R}\big{/}\Sigma\mathcal{M} \cong\text{
mod}\overline{\mathcal{M}}$.
\end{cor}

\begin{proof}
According to Theorem 3.6, we prove the corollary by several steps.

(a) $\overline{\mathcal{M}}$ is a rigid subcategory of the right
triangulated category
$\overline{\mathcal{B}}=\mathcal{B}/\mathcal{I}$.

Let $\Sigma$ be the shift functor of $\overline{\mathcal{B}}$, then
it is easy to see that $\Sigma
\overline{\mathcal{M}}=\overline{\Sigma\mathcal{M}}$. We claim that
Hom$_{\overline{\mathcal{B}}}(\overline{\mathcal{M}},\Sigma\overline{\mathcal{M}})=0$.
In fact, for any $\overline{f}\in\hom_{\overline{\mathcal{B}}}(M,
Y)$, where $M\in\overline{\mathcal{M}}$ and
$Y\in\Sigma\overline{\mathcal{M}}$. There is a short exact sequence
$0\rightarrow M'\xrightarrow{i} I\xrightarrow{d} Y\rightarrow0$,
where $M'\in\mathcal{M}$, $I\in\mathcal{I}$.  Since $\mathcal{M}$ is
rigid in $\mathcal{B}$, applying Hom$_\mathcal{B}(M,-)$ to the short
exact sequence, we have an exact sequence
$$0\rightarrow \mbox{Hom}(M,M')\xrightarrow {i_\ast}\mbox{Hom}(M,I)\xrightarrow {d_\ast}\mbox{Hom}(M,Y)\rightarrow0.$$
So $d$ is a right $\mathcal{M}$-approximation of $Y$. Thus, $f$
factors through $I$, hence $\overline{f}=0$.

(b)
 $\overline{{\mathcal{M}}_{R}}=\overline{\mathcal{M}}\ast\Sigma\overline{\mathcal{M}}$.

It follows from Example 3.3.

(c)
$\mathcal{M}_{R}\big{/}\Sigma\mathcal{M}\cong\overline{\mathcal{M}_{R}}\big{/}\Sigma\overline{\mathcal{M}}$.

It follows from Lemma 2.2 since
$\mathcal{I}\subset\Sigma\mathcal{M}\subset\mathcal{M}_{R}$ and
$\Sigma \overline{\mathcal{M}}=\overline{\Sigma\mathcal{M}}$.

(d) $\Sigma|_{\overline{\mathcal{M}}}$ is fully faithful.

For any $M',M''\in\mathcal{M}$, there exist two short exact
sequences $0\rightarrow
M'\xrightarrow{i_{M'}}I_{M'}\xrightarrow{d_{M'}}\Sigma
M'\rightarrow0$ and $0\rightarrow
M''\xrightarrow{i_{M''}}I_{M''}\xrightarrow{d_{M''}}\Sigma
M''\rightarrow0$, where $I_{M'},I_{M''}\in\mathcal{I}$, and
$d_{M'}$,$d_{M''}$ are right $\mathcal{M}$-approximations.

 For any morphism $\alpha:\Sigma M'\rightarrow\Sigma M''$ in
$\mathcal{B}$, since $d_{M''}$ is a right
$\mathcal{M}$-approximation and
$I_{M'}\in\mathcal{I}\subset\mathcal{M}$ , we have the following
commutative diagram with exact rows in $\mathcal{B}$
$$\xymatrix{
0\ar[r] & M'\ar@{-->}[d]^{m}\ar[r]^{i_{M'}} &
I_{M'}\ar[r]^{d_{M'}}\ar@{-->}[d]^{j} & \Sigma
M'\ar[r]\ar[d]^{\alpha}
& 0 \\
0\ar[r] & M''\ar[r]^{i_{M''}} & I_{M''}\ar[r]^{d_{M''}} & \Sigma
M''\ar[r] & 0.}
$$
Hence we have $\overline{\alpha}=\Sigma\overline{m}$ by the
definition of $\Sigma$, thus $\Sigma|_{\mathcal{M}}$ is full.

 For any morphism $f:M'\rightarrow M''$ in $\mathcal{B}$, Since
$I_{M'}$ is an injective object, we have the following commutative
diagram of short exact sequences %with exact rows in $\mathcal{B}$:
$$\xymatrix{
0\ar[r] & M'\ar[d]^{f}\ar[r]^{i_{M'}} &
I_{M'}\ar[r]^{d_{M'}}\ar[d]^{i_{f}} & \Sigma M'\ar[r]\ar[d]^{\Sigma
f}
& 0 \\
0\ar[r] & M''\ar[r]^{i_{M''}} & I_{M''}\ar[r]^{d_{M''}} & \Sigma
M''\ar[r] & 0.}
$$
Suppose $\Sigma \overline{f}=0$, then $\Sigma f$ factors through
some object in $\mathcal{I}$.
 Because $d_{M''}$ is right $\mathcal{M}$-approximation, $\Sigma
f$ factors through $I_{M''}$, i.e. there exists a morphism $a:\Sigma
M'\rightarrow I_{M''}$ such that $\Sigma f=d_{M''}a$. Then
$d_{M''}(i_{f}-ad_{M'})=d_{M''}i_{f}-(\Sigma f)d_{M'}=0$, thus there
exists a morphism $b:I_{M'}\rightarrow M''$ such that
$i_{M''}b=i_{f}-ad_{M'}$, so
$i_{M''}(f-bi_{M'})=i_{M''}f-i_{f}i_{M'}+ad_{M'}i_{M'}=0$. Since
$i_{M''}$ is a monomorphism, $f=bi_{M'}$, thus $f$ factors through
$I_{M'}$. Hence $\overline{f}=0$ and $\Sigma|_{\mathcal{M}}$ is
faithful.

(e) Let
$M'\xrightarrow{\overline{f}}M''\xrightarrow{\overline{g}}X\xrightarrow{\overline{h}}\Sigma
M'$ be a right triangle in $\overline{\mathcal{B}}$ with
$M',M''\in\mathcal{M}$, then $\overline{g}$ is a right
$\overline{\mathcal{M}}$-approximation of $X$.

According to Example 3.3 and $\mathcal{I}\subset\mathcal{M}$, we can
assume that there is a short exact sequence $0\rightarrow
M'\xrightarrow{f} M''\xrightarrow{g}X\rightarrow0$. Since
$\mathcal{M}$ is rigid,
%$(g,-i_f)$ is a right
%$\mathcal{M}$-approximation of $X$. Then
there exists an epimorphism Hom$_{\mathcal{B}}(M,g):$ Hom$
_{\mathcal{B}}(M,M'')\rightarrow$Hom$_{\mathcal{B}}(M,X)$ for any
$M$ in $\mathcal{M}$. Thus we have an epimorphism
Hom$_{\mathcal{\overline{B}}}(M,\overline{g}):$  Hom$
_{\mathcal{\overline{B}}}(M,$ $ M'')\rightarrow$
Hom$_{\mathcal{\overline{B}}}(M,X)$, i.e. $\overline{g}$ is a right
$\overline{\mathcal{M}}$-approximation of $X$.
\end{proof}

%\begin{example} Let $Q:4\rightarrow3\rightarrow2\rightarrow1$,
%$\mathcal{C}=\mathcal{C}_{Q}$, the cluster category of $Q$ whose AR
%quiver is the following:
%$$\xymatrix{
%&&&P_{4}\ar[rd]&&P_{1}[1]\ar[rd]&&&\\
%&&P_{3}\ar[rd]\ar[ru]&&I_{2}\ar[rd]\ar[ru]&&P_{2}[1]\ar[rd]&&\\
%&P_{2}\ar[rd]\ar[ru]&&X\ar[rd]\ar[ru]&&I_{3}\ar[rd]\ar[ru]&&P_{3}[1]\ar[rd]&\\
%P_{1}\ar[ru]&&S_{2}\ar[ru]&&S_{3}\ar[ru]&&I_{4}\ar[ru]&&P_{4}[1].\\
%}
%$$
%We take $\mathcal{Y}=add\{S_{3}\}$. Then by [12],
%$\mathcal{C}/\mathcal{Y}$ is a right triangulated category  and
%$<1>$ is the shift functor. Denote $\Sigma=<1>$. Let
%$\mathcal{M}=add\{P_{1},P_{2},P_{4}\}$. Then $\mathcal{M}$ is a
%rigid subcategory in $\mathcal{C}/\mathcal{Y}$ satisfying
%$\Sigma|_{\mathcal{M}}$ is fully faithful. And (RC2) is natural.
%Hence we have
%$\mathcal{M}\ast\Sigma\mathcal{M}=add\{P_{1},P_{2},P_{4},S_{2},I_{2},I_{3},P_{1}[1],P_{2}[1],P_{4}[1]\}$.
%Then by Theorem 3.5,
%$\mathcal{M}\ast\Sigma\mathcal{M}\big{/}\Sigma\mathcal{M}=add\{P_{1},P_{2},P_{4},S_{2},I_{2},I_{3}\}$
%is equivalent to the module category over $\mathcal{M}$, which is an
%abelian category. In fact
%$\mathcal{M}\ast\Sigma\mathcal{M}\big{/}\Sigma\mathcal{M}$ is
%equivalent to the module category over path algebra $kQ'$, where
%$Q':3\rightarrow2\rightarrow1$.
%\end{example}

\section{Subquotient categories of left triangulated categories}

The definition of left triangulated category is dual to right
triangulated category. For convenience,  we recall the definition
and some facts.

\begin{defn} ([7])
A left triangulated category is a triple
$(\mathcal{C},\Omega,\lhd)$, or simply $\mathcal{C}$, where:

(a)\quad $\mathcal{C}$ is an additive category.

(b)\quad $\Omega:\mathcal{C}\rightarrow\mathcal{C}$ is an additive
functor, called the shift functor of $\mathcal{C}$.

(c)\quad $\lhd$ is a class of sequences of three morphisms of the
form $\Omega Z\xrightarrow{x}X\xrightarrow{y}Y\xrightarrow{z}Z$,
called left triangles, and satisfying the following axioms:

(LTR0)\quad If $\Omega
Z\xrightarrow{x}X\xrightarrow{y}Y\xrightarrow{z}Z$ is a left
triangle, and $\Omega
Z'\xrightarrow{x'}X'\xrightarrow{y'}Y'\xrightarrow{z'}Z'$ is a
sequence of morphisms such that there exists a commutative diagram
in $\mathcal{C}$
$$\xymatrix{
\Omega Z\ar[r]^{x}\ar[d]^{\Omega h} & X\ar[r]^{y}\ar[d]^{f} &
Y\ar[r]^{z}\ar[d]^{g} & Z\ar[d]^{h} \\
\Omega Z'\ar[r]^{x'} & X'\ar[r]^{y'} & Y'\ar[r]^{z'} & Z', }
$$
where $f,g,h$ are isomorphisms, then $\Omega
Z'\xrightarrow{x'}X'\xrightarrow{y'}Y'\xrightarrow{z'}Z'$ is also a
left triangle.

(LTR1)\quad For any $X\in\mathcal{C}$, the sequence
$0\xrightarrow{}X\xrightarrow{1_{X}}X\xrightarrow{}0$ is a left
triangle. And for every morphism $z:Y\rightarrow Z$ in
$\mathcal{C}$, there exists a left triangle $\Omega
Z\xrightarrow{x}X\xrightarrow{y}Y\xrightarrow{z}Z$.

(LTR2)\quad If $\Omega
Z\xrightarrow{x}X\xrightarrow{y}Y\xrightarrow{z}Z$ is a left
triangle, then so is $\Omega Y\xrightarrow{-\Omega z}\Omega
Z\xrightarrow{x}X\xrightarrow{y}Y$.

(LTR3)\quad For any two left triangles $\Omega
Z\xrightarrow{x}X\xrightarrow{y}Y\xrightarrow{z}Z$ and $\Omega
Z'\xrightarrow{x'}X'\xrightarrow{y'}Y'\xrightarrow{z'}Z'$, and any
two morphisms $g:Y\rightarrow Y'$, $h:Z\rightarrow Z'$ such that
$hz=z'g$, there exists $f:X\rightarrow X'$ making the following
diagram commutative
$$\xymatrix{
\Omega Z\ar[r]^{x}\ar[d]^{\Omega h} & X\ar[r]^{y}\ar@{-->}[d]^{f} &
Y\ar[r]^{z}\ar[d]^{g} & Z\ar[d]^{h} \\
\Omega Z'\ar[r]^{x'} & X'\ar[r]^{y'} & Y'\ar[r]^{z'} & Z' }
$$

(LTR4)\quad For any two left triangles $\Omega
Z\xrightarrow{x}X\xrightarrow{y}Y\xrightarrow{z}Z$ and $\Omega
Z'\xrightarrow{x'}X'\xrightarrow{y'}Z\xrightarrow{z'}Z'$, there
exists a commutative diagram
$$\xymatrix{
                              & \Omega Y'\ar[r]^{\Omega y'}\ar[d]^{x\cdot\Omega y'} & \Omega Z\ar[d]^{x} &  \\
                              & X\ar@{=}[r]\ar[d]^{g}  & X\ar[d]^{y}  &           \\
\Omega Z'\ar[r]^{u}\ar@{=}[d] & X'\ar[r]^{v}\ar[d]^{h} & Y\ar[r]^{z'\cdot z}\ar[d]^{z} & Z'\ar@{=}[d]                    \\
\Omega Z'\ar[r]^{x'}          & Y'\ar[r]^{y'} & Z\ar[r]^{z'} & Z', }
$$
where the third row and the second column are left triangles.
\end{defn}

\begin{example}
A triangulated category is a left triangulated category.
\end{example}

\begin{example}
Let $\mathcal{B}$ be an exact category with enough projectives.
Denote by $\mathcal{P}$ the subcategory of $\mathcal{B}$ consisting
of projectives. Then the quotient category
$\underline{\mathcal{B}}=\mathcal{B}/\mathcal{P}$ is a left
triangulated category.
\end{example}

By (LTR0) and (LTR2), we have the following easy lemma.

\begin{lem}
Let $\Omega Z\xrightarrow{x}X\xrightarrow{y}Y\xrightarrow{z}Z$ be a
left triangle, then so is $\Omega Y\xrightarrow{\Omega z}\Omega
Z\xrightarrow{x}X\xrightarrow{-y}Y$.
\end{lem}

\begin{lem}(cf. [8])
Let $\mathcal{C}$ be a left triangulated category. Then for any left
triangle $\Omega Z\xrightarrow{x}X\xrightarrow{y}Y\xrightarrow{z}Z$
and any object $U$ of $\mathcal{C}$, there exists an exact sequence

$\cdots\rightarrow$Hom$_{\mathcal{C}}(U,\Omega
Z)\xrightarrow{x_{\ast}}$Hom$_{\mathcal{C}}(U,X)\xrightarrow{y_{\ast}}$
Hom$_{\mathcal{C}}(U,Y)\xrightarrow{z_{\ast}}$Hom$_{\mathcal{C}}(U,Z)$.
\end{lem}

\begin{defn}
Let $\mathcal{C}$ be a left triangulated category. A subcategory
$\mathcal{M}$ of $\mathcal{C}$ is called a rigid subcategory if
Hom$_{\mathcal{C}}(\Omega\mathcal{M},\mathcal{M})=0$.
\end{defn}

Let $\mathcal{M}$ be a rigid subcategory of $\mathcal{C}$. Denote by
$\Omega\mathcal{M}\ast\mathcal{M}$ the subcategory of objects $X$ in
$\mathcal{C}$ such that there exist left triangles $\Omega
M_{1}\rightarrow X\rightarrow M_{0}\rightarrow M_{1}$, where
$M_{0},M_{1}\in\mathcal{M}$. Now we consider the functor
$H:\Omega\mathcal{M}\ast\mathcal{M}\rightarrow \text{Mod
}\mathcal{M}$ defined by
$H(X)=\text{Hom}_{\mathcal{C}}(\Omega(-),X)|_{\mathcal{M}}$.
 %Since $\mathcal{M}$ is a rigid subcategory, $H(\mathcal{M}) $=0.
%Hence by the universal property of the projection functor
%$\pi:\Omega\mathcal{M}\ast\mathcal{M}\rightarrow\Omega\mathcal{M}\ast\mathcal{M}\big{/}\mathcal{M}$,
%$H$ can induce a functor $
%\overline{H}:\Omega\mathcal{M}\ast\mathcal{M}\big{/}\mathcal{M}\rightarrow
%\text{Mod }\mathcal{M}$.

\begin{lem}
Let $(\mathcal{C},\Omega,\lhd)$ be a left triangulated category and
$\mathcal{M}$ be a rigid subcategory of $\mathcal{C}$. If
$\Omega|_{\mathcal{M}}$ is fully faithful, then for any left
triangle $\Omega
M_{1}\xrightarrow{f}X\xrightarrow{g}M_{0}\xrightarrow{h}M_{1}$ where
$M_0, M_1\in\mathcal{M}$, there is an exact sequence in
Mod$\mathcal{M}$ $$\mbox{Hom}
_{\mathcal{M}}(-,M_{0})\xrightarrow{\text{Hom}_{\mathcal{M}}(-,h)}
\mbox{Hom}_{\mathcal{M}}(-,M_{1})\rightarrow
%{\text{Hom}_{\mathcal{M}}(-,g)}
H(X)\rightarrow0.$$ Thus,
%$\text{Hom}_{\mathcal{C}}(-,X)|_{\mathcal{M}}\in\text{mod
%}\mathcal{M}$.then for any object $X$ in
%$\Omega\mathcal{M}\ast\mathcal{M}$,
$H(X)\in$ mod$\mathcal{M}$.
\end{lem}

\begin{proof}
For any $X\in\Omega\mathcal{M}\ast\mathcal{M}$, there exists a left
triangle $\Omega
M_{1}\xrightarrow{f}X\xrightarrow{g}M_{0}\xrightarrow{h}M_{1}$,
where $M_{0},M_{1}\in\mathcal{M}$. Then $\Omega
M_{0}\xrightarrow{\Omega h}\Omega
M_{1}\xrightarrow{f}X\xrightarrow{-g}M_{0}$ is a left triangle by
Lemma 4.4. Thus there exists an exact sequence by Lemma 4.5
$$\text{Hom}_{\mathcal{C}}(\Omega M,\Omega M_{0})\xrightarrow
{(\Omega h)_{\ast}}\text{Hom}_{\mathcal{C}}(\Omega M,\Omega
M_{1})\xrightarrow{f_\ast}$$$$\text{Hom}_{\mathcal{C}}(\Omega
M,X)\rightarrow \text{Hom}_{\mathcal{C}}(\Omega M,M_0)=0.$$ Since
$\Omega|_{\mathcal{M}}$ is fully faithful, we have the following
commutative diagram with exact rows
$$\xymatrix{
\text{Hom}_{\mathcal{C}}(M,M_{0})\ar[d]\ar[r]^{h_{\ast}}&
 \text{Hom}_{\mathcal{C}}(M,M_{1})\ar[d]\ar[r]&
\text{Hom}_{\mathcal{C}}(\Omega M,X)\ar[r]\ar@{=}[d]& 0
\\
\text{Hom}_{\mathcal{C}}(\Omega M,\Omega M_{0})\ar[r]^{(\Omega
h)_{\ast}} & \text{Hom}_{\mathcal{C}}(\Omega M,\Omega M_{1})\ar[r]&
\text{Hom}_{\mathcal{C}}(\Omega M,X)\ar[r]& 0  ,}
$$
where $M\in\mathcal{M}$ and the vertical morphisms are isomorphisms.
Thus we have an exact sequence in Mod$\mathcal{C}$
$$\mbox{Hom}_{\mathcal{M}}(-,M_{0})
\xrightarrow{\text{Hom}_{\mathcal{M}}(-,h)}\mbox{Hom}_{\mathcal{M}}(-,M_{1})\rightarrow
H(X)\rightarrow0.$$ So $H(X)\in$ mod$\mathcal{M}.$
\end{proof}

%By the lemma above, $\overline{H}$ can be seen as a functor from
%$\Omega\mathcal{M}\ast\mathcal{M}$ to mod $\mathcal{M}$.

\begin{thm}
Let $\mathcal{C}$ be a left triangulated category, and $\mathcal{M}$
be a rigid subcategory of $\mathcal{C}$ satisfying:

(LC1)\quad $\Omega$ is fully faithful when it is restricted to
$\mathcal{M}$.

(LC2)\quad Let $\Omega M_{1}\xrightarrow{f}X\xrightarrow{g}M_{0}
\xrightarrow{h}M_{1}$ be a left triangle, where $M_0, M_1\in
\mathcal{M}$. Let $Y\in\Omega\mathcal{M}\ast\mathcal{M}$ and a
morphism $t:X\rightarrow Y$  such that $tf=0$, then $t$ factors
through $g$.

Then there exists an equivalence of categories
$\Omega\mathcal{M}\ast\mathcal{M}/\mathcal{M}\cong\mbox{mod}\mathcal{M}$.
\end{thm}

\begin{proof}
According to Lemma 4.7, we have a functor
$H:\Omega\mathcal{M}\ast\mathcal{M}\rightarrow \text{mod
}\mathcal{M}$.

% $\forall X,Y\in\Omega\mathcal{M}\ast\mathcal{M}$,
%there exist two left triangles $\Omega
%M_{1}\xrightarrow{f_{1}}X\xrightarrow{g_{1}}M_{0}
%\xrightarrow{h_{1}}M_{1}$ and $\Omega
%N_{1}\xrightarrow{f_{2}}Y\xrightarrow{g_{2}}N_{0}
%\xrightarrow{h_{2}}N_{1}$, where
%$M_{0},M_{1},N_{0},N_{1}\in\mathcal{M}$. Then by Remark 4.2, there
%exist two left triangles $\Omega M_{0}\xrightarrow{\Omega
%h_{1}}\Omega M_{1}\xrightarrow{f_{1}}X\xrightarrow{-g_{1}}M_{0}$ and
%$\Omega N_{0}\xrightarrow{\Omega h_{2}}\Omega
%N_{1}\xrightarrow{f_{2}}Y\xrightarrow{-g_{2}}N_{0}$.

 Firstly, we show that $H$ is dense.

For any object $G\in$ mod $\mathcal{M}$, there exists an exact
sequence
$$\text{Hom}_{\mathcal{M}}(-,M')\xrightarrow{\alpha}\text{Hom}_{\mathcal{M}}(-,M'')\rightarrow
G\rightarrow0$$ with $M',M''\in\mathcal{M}$. By Yoneda's Lemma,
there exists a morphism $h:M'\rightarrow M''$ such that
$\alpha$=Hom$_{\mathcal{M}}(-,h)$. Then by (LTR1), there exists a
left triangle $\Omega
M''\xrightarrow{f}Z\xrightarrow{g}M'\xrightarrow{h}M''$.  Hence by
Lemma 4.7, there exists an exact sequence
$$\text{Hom}_{\mathcal{M}}(-,
M')\xrightarrow{\alpha}\text{Hom}_{\mathcal{M}}(-,M'')\rightarrow
H(Z)\rightarrow0,$$ so $G$=Coker$\alpha\cong H(Z)$. Hence $H$ is
dense.

Secondly, we show that $H$ is full.

For any morphism $\beta:H(X)\rightarrow H(Y)$ in mod$\mathcal{M}$.
By Lemma 4.7 and because Hom$_{\mathcal{M}}(-,M_{1})$ is a
projective object of mod$\mathcal{M}$, we have the following
commutative diagram with exact rows in Mod$\mathcal{M}$
$$\xymatrix{
\text{Hom}_{\mathcal{M}}(-,M_{0})\ar[rr]^{\text{Hom}_{\mathcal{M}}(-,h_{1})}\ar[d]^{\gamma_{0}}&
& \text{Hom}_{\mathcal{M}}(-,M_{1})\ar[r]\ar[d]^{\gamma_{1}} &
H(X)\ar[d]^{\beta}\ar[r] & 0 \\
\text{Hom}_{\mathcal{M}}(-,N_{0})\ar[rr]^{\text{Hom}_{\mathcal{M}}(-,h_{2})}&
& \text{Hom}_{\mathcal{M}}(-,N_{1})\ar[r]& H(Y)\ar[r] & 0.} $$ By
Yoneda's Lemma, for $i=0,1$, there exists a morphism
$m_{i}:M_{i}\rightarrow N_{i}$ such that
$\gamma_{i}$=Hom$_{\mathcal{M}}(-, m_{i})$ and
$m_{1}h_{1}$=$h_{2}m_{0}$. Hence by (LTR3), we have the following
commutative diagram of left triangles
$$\xymatrix{
\Omega M_{1}\ar[r]^{f_{1}}\ar[d]^{\Omega m_{1}} &
X\ar[r]^{g_{1}}\ar@{-->}[d]^{s} &
M_{0}\ar[r]^{h_{1}}\ar[d]^{m_{0}} & M_{1}\ar[d]^{m_{1}} \\
\Omega N_{1}\ar[r]^{f_{2}} & Y\ar[r]^{g_{2}} & N_{0}\ar[r]^{h_{2}} &
N_{1}. }
$$
According to the proof of Lemma 4.7, for any object
$M\in\mathcal{M}$, we have the following commutative diagram with
exact columns.
%$${\tiny
%\xymatrix{ & & (M,M_{0})\ar[ld]\ar[dd]_{
%m_{0\ast}}\ar[rr]^{h_{1\ast}} & &
%(M,M_{1})\ar[ld]\ar[rr]\ar[dd]_{m_{1\ast}}  & &
%(\Omega M,X)\ar@{=}[ld]\ar[r]\ar[dd]_{s_{\ast}} & 0 \\
% & (\Omega M,\Omega M_{0})\ar[rr]^{h_{2\ast}}\ar[dd]_{(\Omega m_0)_\ast}  & & (\Omega M,\Omega M_{1})
% \ar[rr]\ar[dd]_{(\Omega m_1)_\ast}  & &(\Omega M,X)\ar[rr]\ar[dd]_{s_{\ast}} & & 0 \\
% & & \Omega N''\ar[ld]_{f_{2}}\ar[rr]^{i_{N''}} & &
%  P_{N''}\ar[ld]_{p_{N}}\ar[rr]^{d_{N''}}  & &
%N''\ar@{=}[ld]\ar[r] & 0 \\
% & Y\ar[rr]^{g_{2}}  & & N'\ar[rr]^{h_{2}}  & & N''\ar[rr] &
%& 0,}}
%$$
$${\tiny\xymatrix{
&\text{Hom}_{\mathcal{C}}(M,M_{0})\ar[dd]_{h_{1\ast}}\ar[dl]_{m_{0\ast}}\ar[rr]&&\text{Hom}_{\mathcal{C}}(\Omega
M,\Omega
M_{0})\ar[dl]_{(\Omega m_{0})_{\ast}}\ar[dd]_{(\Omega h_{1})_{\ast}}\\
\text{Hom}_{\mathcal{C}}(M,N_{0})\ar[rr]\ar[dd]_{h_{2\ast}}&&\text{Hom}_{\mathcal{C}}(\Omega M,\Omega N_{0})\ar[dd]_{(\Omega h_{2})_{\ast}}&\\
&\text{Hom}_{\mathcal{C}}(M,M_{1})\ar[dl]_{m_{1\ast}}\ar[rr]\ar[dd]&&
\text{Hom}_{\mathcal{C}}(\Omega M,\Omega M_{1})\ar[dl]_{(\Omega m_{1})_{\ast}}\ar[dd]\\
\text{Hom}_{\mathcal{C}}(M,N_{1})\ar[rr]\ar[dd]&&\text{Hom}_{\mathcal{C}}(\Omega
M,\Omega N_{1})\ar[dd]&\\
&\text{Hom}_{\mathcal{C}}(\Omega M,X)\ar[ld]_{s_{\ast}}\ar@{=}[rr]\ar[dd]&&\text{Hom}_{\mathcal{C}}(\Omega M,X)\ar[ld]_{s_{\ast}}\ar[dd]\\
\text{Hom}_{\mathcal{C}}(\Omega M,Y)\ar@{=}[rr]\ar[d] &&\text{Hom}_{\mathcal{C}}(\Omega M,Y)\ar[d] &\\
0&0&0&0. }}
$$
Thus we have the following commutative diagram with exact rows in
Mod $\mathcal{M}$
$${\footnotesize\xymatrix{
\text{Hom}_{\mathcal{M}}(-,M_{0})\ar[rr]^{\text{Hom}_{\mathcal{M}}(-,h_{1})}\ar[d]^{\gamma_{0}}&
& \text{Hom}_{\mathcal{M}}(-,M_{1})\ar[r]\ar[d]^{\gamma_{1}} &
H(X)\ar[d]^{H(s)}\ar[r] & 0 \\
\text{Hom}_{\mathcal{M}}(-,N_{0})\ar[rr]^{\text{Hom}_{\mathcal{M}}(-,h_{2})}&
& \text{Hom}_{\mathcal{M}}(-,N_{1})\ar[r]& H(Y)\ar[r] & 0.}}
$$
So $\beta=H(s)$. Hence $H$ is full.

At last, let $X, Y$ be objects of
$\Omega\mathcal{M}\ast\mathcal{M}$. We have a left triangle $\Omega
M_{1}\xrightarrow{f}X\xrightarrow{g}M_{0} \xrightarrow{h}M_{1}$,
where $M_0, M_1\in \mathcal{M}$. Let  $t: X\rightarrow Y$ be a
morphism with $H(t)=0$, then $tf=0$. Thus $t$ factors through $M_0$
by (LC2). So
$\Omega\mathcal{M}\ast\mathcal{M}/\mathcal{M}\cong\text{mod}
\mathcal{M}$ by Lemma 2.2.
\end{proof}

Since a triangulated category is a left triangulated category such
that the shift functor is an equivalence, the conditions (LC1) and
(LC2) holds automatically. Thus we have the following corollary.

\begin{cor}
Let $\mathcal{C}$ be a triangulated category with the shift functor
$T$ and $\mathcal{M}$ be a rigid subcategory of $\mathcal{C}$, then
$T^{-1}\mathcal{M}\ast\mathcal{M}/\mathcal{M}\cong\text{mod}
\mathcal{M}$.
\end{cor}

%We denote $T^{-1}=\Omega$. Then we can easily prove that
%($\mathcal{C},\Omega,\varepsilon$) is a left triangulated category,
%i.e. every triangle $\Omega Z\rightarrow X\rightarrow Y\rightarrow
%Z$ is also a left triangle. Let $\mathcal{M}$ be a rigid subcategory
%of triangulated category $\mathcal{C}$. Then $\mathcal{M}$ is a
%rigid subcategory of left triangulated category $\mathcal{C}$.
%Denote by $\Omega\mathcal{M}\ast\mathcal{M}$ the subcategory of
%objects $X$ in $\mathcal{C}$ such that there exist left triangles
%$\Omega M_{1}\rightarrow X\rightarrow M_{0}\rightarrow M_{1}$, where
%$M_{0},M_{1}\in\mathcal{M}$. Since
%$\Omega:\mathcal{C}\rightarrow\mathcal{C}$ is auto-equivalent, so
%$\Omega|_{\mathcal{M}}$ is fully faithful and
%$\overline{H}:\Omega\mathcal{M}\ast\mathcal{M}\big{/}\mathcal{M}\rightarrow\text{mod
%}\mathcal{M}$ is faithful. Hence by Corollary 3.2.5, there exists an
%equivalence of categories:
%$\Omega\mathcal{M}\ast\mathcal{M}\big{/}\mathcal{M}\cong\text{mod
%}\mathcal{M}\cong\text{mod }\Omega\mathcal{M}$, which coincides with
%Proposition 6.2 in [4].

\begin{cor} ([5], Theorem 3.2)
Let $\mathcal{B}$ be an exact category which contains enough
projectives, and $\mathcal{M}$ be a rigid subcategory of
$\mathcal{B}$ containing all projectives. Denote by $\mathcal{P}$
the subcategory of projectives, and by $\underline{\mathcal{M}}$ the
quotient category $\mathcal{M}/\mathcal{P}$. Denote by
$\mathcal{M}_{L}$ the subcategory of objects $X$ in $\mathcal{B}$
such that there exist short exact sequences $0\rightarrow
X\rightarrow M_{0}\rightarrow M_{1}\rightarrow0$, where
$M_{0},M_{1}\in\mathcal{M}$. Then $\mathcal{M}_{L}\big{/}\mathcal{M}
\cong$ mod$\underline{\mathcal{M}}$.
\end{cor}

\begin{proof}
Similar to the proof of Corollary 3.10, we can prove that
$\underline{\mathcal{M}}$ is a rigid subcategory of the left
triangulated category $\underline{\mathcal{B}}$, and
$\underline{\mathcal{M}_{L}}=\Omega\underline{\mathcal{M}}\ast\underline{\mathcal{M}}$,
and
$\mathcal{M}_{L}\big{/}\mathcal{M}\cong\underline{\mathcal{M}_{L}}\big{/}\underline{\mathcal{M}}$,
and $\Omega|_{\underline{\mathcal{M}}}$ is fully faithful. To end
the proof, we only need to show that $\underline{\mathcal{M}}$
satisfies the condition (LC2).

 In fact, let $\Omega
M''\xrightarrow{\underline{f_{1}}}X\xrightarrow{\underline{g_{1}}}M'\xrightarrow{\underline{h_{1}}}M''$
be a left triangle in $\underline{\mathcal{B}}$, where
$M',M''\in\mathcal{M}$. Since $\mathcal{P}\subset\mathcal{M}$, we
can assume that $0\rightarrow X\xrightarrow{g_1
}M'\xrightarrow{h_1}M''\rightarrow0$ is a short exact sequence.
%so there exists a short exact sequence $0\rightarrow
%X\xrightarrow{g_{1}}M'\xrightarrow{h_{1}}M''\rightarrow0$.
 Let
$t:X\rightarrow Y$ be a morphism satisfying $\underline{tf_{1}}=0$,
where $Y\in\mathcal{M}_{L}$. Then there exists a short exact
sequence $0\rightarrow
Y\xrightarrow{g_{2}}N'\xrightarrow{h_{2}}N''\rightarrow0$, where
$N',N''\in\mathcal{M}$. Since $\mathcal{M}$ is rigid, it is easy to
see that $g_{1}$ is a left $\mathcal{M}$-approximation, then we have
the following commutative diagram with exact rows in $\mathcal{B}$
$$\xymatrix{
0\ar[r] & X\ar[d]_{t}\ar[r]^{g_{1}} & M'\ar[r]^{h_{1}}\ar[d]_{m_{1}}
& M''\ar[r]\ar[d]_{m_{2}}
& 0 \\
0\ar[r] & Y\ar[r]^{g_{2}} & N'\ar[r]^{h_{2}} & N''\ar[r] & 0.}
$$
The lower exact sequence induces a left triangle $\Omega
N''\xrightarrow{\underline{f_{2}}}Y\xrightarrow{\underline{g_{2}}}N'\xrightarrow{\underline{h_{2}}}N''$.
We claim that
$\underline{tf_{1}}=\underline{f_{2}}\Omega\underline{m_{2}}$. In
fact,  we have the following diagram with exact rows in
$\mathcal{B}$
$${\footnotesize\xymatrix{
0\ar[rr] & & \Omega M''\ar[ld]_{f_{1}}\ar[dd]_{\Omega
m_{2}}\ar[rr]^{i_{M''}} & &
 P_{M''}\ar[ld]_{p_{M}}\ar[rr]^{d_{M''}}\ar[dd]_{p}  & &
M''\ar@{=}[ld]\ar[r]\ar[dd]_{m_{2}} & 0 \\
0\ar[r] & X\ar[rr]^{g_{1}}\ar[dd]_{t}  & & M'\ar[rr]^{h_{1}}\ar[dd]_{m_{1}}  & & M''\ar[rr]\ar[dd]_{m_{2}} & & 0 \\
0\ar[rr] & & \Omega N''\ar[ld]_{f_{2}}\ar[rr]^{i_{N''}} & &
  P_{N''}\ar[ld]_{p_{N}}\ar[rr]^{d_{N''}}  & &
N''\ar@{=}[ld]\ar[r] & 0 \\
0\ar[r] & Y\ar[rr]^{g_{2}}  & & N'\ar[rr]^{h_{2}}  & & N''\ar[rr] &
& 0,}}
$$
where $P_{M''},P_{N''}\in\mathcal{P}$,  and all squares are
commutative except the left one and the middle one.
%$\xymatrix{P_{M''}\ar[r]^{p_{M}}\ar[d]_{p} & M'\ar[d]_{m_{1}}\\
%P_{N''}\ar[r]_{p_{N}} & N'
%}$ and $\xymatrix{\Omega M''\ar[r]^{f_{1}}\ar[d]_{\Omega m_{2}} & X.\ar[d]_{t}\\
%\Omega N''\ar[r]_{f_{2}} & Y }$
Since $h_{2}(m_{1}p_{M}-p_{N}p)=m_{2}d_{M''}-m_{2}d_{M''}=0$, there
exists a morphism $q:P_{M''}\rightarrow Y$ such that
$g_{2}q=m_{1}p_{M}-p_{N}p$. Then $g_{2}(tf_{1}-f_{2}\Omega
m_{2}-qi_{M''})=(m_{1}p_{M}-p_{N}p)i_{M''}-(m_{1}p_{M}-p_{N}p)i_{M''}=0$.
Since $g_{2}$ is a monomorphism, we get $tf_{1}-f_{2}\Omega
m_{2}=qi_{M''}$. Thus
$\underline{tf_{1}}=\underline{f_{2}}\Omega\underline{m_{2}}$. Hence
we have the following commutative diagram of left triangles in
$\underline{\mathcal{B}}$
$$\xymatrix{
\Omega M''\ar[d]_{\Omega\underline{m_{2}}}\ar[r]^{\underline{f_{1}}}
& X\ar[d]_{\underline{t}}\ar[r]^{\underline{g_{1}}} &
M'\ar[r]^{\underline{h_{1}}}\ar[d]_{\underline{m_{1}}} & M''\ar[d]_{\underline{m_{2}}}   \\
\Omega N''\ar[r]^{\underline{f_{2}}} & Y\ar[r]^{\underline{g_{2}}} &
N'\ar[r]^{\underline{h_{2}}} & N''. }
$$
By Lemma 4.4, we have the following commutative diagram of left
triangles in $\underline{\mathcal{B}}$
$$\xymatrix{
\Omega
M'\ar[d]_{\Omega\underline{m_{1}}}\ar[r]^{\Omega\underline{h_{1}}} &
\Omega M''\ar[d]_{\Omega\underline{m_{2}}}\ar[r]^{\underline{f_{1}}}
&
X\ar[r]^{\underline{-g_{1}}}\ar[d]_{\underline{t}} & M'\ar[d]_{\underline{m_{1}}}   \\
\Omega N'\ar[r]^{\Omega\underline{h_{2}}} & \Omega
N''\ar[r]^{\underline{f_{2}}} & Y\ar[r]^{\underline{-g_{2}}} & N'. }
$$
Since $\underline{f_{2}}\Omega
\underline{m_{2}}=\underline{tf_{1}}=0$, there exists a morphism
$n':\Omega M''\rightarrow\Omega N'$ such that
$\Omega\underline{m_{2}}=(\Omega\underline{h_{2}})\underline{n'}$.
Because $\Omega|_{\mathcal{M}}$ is fully faithful, there exists a
morphism $n_{1}:M''\rightarrow N'$ such that
$\underline{n'}=\Omega\underline{n_{1}}$ and
$\underline{m_{2}}=\underline{h_{2}n_{1}}$. Hence $m_{2}-h_{2}n_{1}$
factors through $P\in\mathcal{P}$. Since $h_{2}$ is a epimorphism,
we have the following commutative diagram in $\mathcal{B}$:
$$\xymatrix{
  & &
M''\ar[ld]_{a}\ar[dd]^{m_{2}-h_{2}n_{1}}
&  \\
 & P\ar[ld]_{c}\ar[rd]^{b} & \\
 N'\ar[rr]^{h_{2}} & & N''. & }
$$
Let $n=ca+n_{1}$. Then $m_{2}=h_2n_1+ba=h_2n_1+h_2ca=h_{2}n$ and
$\underline{n}=\underline{n_{1}}$.
%Hence for the following
%commutative diagram with exact rows in $\mathcal{B}$:
%$$\xymatrix{
%0\ar[r] & X\ar[d]_{t}\ar[r]^{g_{1}} & M'\ar[r]^{h_{1}}\ar[d]_{m_{1}}
%& M''\ar[r]\ar[d]^{m_{2}}
%& 0 \\
%0\ar[r] & Y\ar[r]^{g_{2}} & N'\ar[r]^{h_{2}} & N''\ar[r] & 0,}
%$$
Since $h_{2}(m_{1}-nh_{1})=h_{2}m_{1}-m_{2}h_{1}=0$, there exists a
morphism $s:M'\rightarrow Y$ such that $g_{2}s=m_{1}-nh_{1}$. Hence
$g_{2}(t-sg_{1})=g_{2}t-m_{1}g_{1}+nh_{1}g_{1}=0$. Because  $g_{2}$
is a monomorphism, $t=sg_{1}$, i.e. $t$ factors through $g_{1}$.
Hence $\underline{t}$ factors through $\underline{g_{1}}$ in
$\underline{\mathcal{B}}$.
\end{proof}

 %We have $\mathcal{P}\subset\mathcal{M}\subset\mathcal{M}_{L}$.

\par

\end{document}